\documentclass[11pt,twoside]{amsart} 
\usepackage{amssymb,color,tikz}
\usetikzlibrary{matrix} \nonstopmode \textwidth=16.00cm

\usepackage{xcolor}
\usepackage{blindtext}
\usepackage{multicol}

\numberwithin{equation}{section} 

\newcommand{\CC}{\mathbb{C}}
\newcommand{\mA}{\mathbb{A}}
\newcommand{\NN}{\mathbb{N}}

\newcommand {\PP}{\mathbb{P}}

\newcommand {\sT}{\mathcal{T}}
\newcommand {\sO}{\mathcal{O}}
\newcommand {\sJ}{\mathcal{J}}
\newcommand {\sL}{\mathcal{L}}
\newcommand {\sC}{\mathcal{C}}
\newcommand{\hess}{\text {Hess}}
 \DeclareMathOperator{\Proj}{Proj}
 \def\cocoa{{\hbox{\rm C\kern-.13em
      o\kern-.07em C\kern-.13em o\kern-.15em A}}}
\newtheorem{theorem}{Theorem}[section]
\newtheorem{lemma}[theorem]{Lemma}
\newtheorem{proposition}[theorem]{Proposition}

 \theoremstyle{definition}
\newtheorem{definition}[theorem]{Definition} \theoremstyle{remark}
\newtheorem{remark}[theorem]{Remark}
\newtheorem{example}[theorem]{Example}

\newtheorem{thmx}{Theorem}

\definecolor{MyDarkGreen}{cmyk}{0.7,0,1,0}

\begin{document}

\title[Jacobian schemes of conic-line arrangements \\
and eigenschemes]
{Jacobian schemes of conic-line arrangements \\
and eigenschemes}
 \author[V. Beorchia]{Valentina Beorchia} 
 \address{Dipartimento di Matematica e Geoscienze, Universit\`a di
Trieste, Via Valerio 12/1, 34127 Trieste, Italy}
 \email{beorchia@units.it, 
 ORCID 0000-0003-3681-9045}.
 \author[R.\ M.\ Mir\'o-Roig]{Rosa M.\ Mir\'o-Roig} 
 \address{Facultat de
 Matem\`atiques i Inform\`atica, Universitat de Barcelona, Gran Via des les
 Corts Catalanes 585, 08007 Barcelona, Spain} \email{miro@ub.edu, ORCID 0000-0003-1375-6547}

\thanks{The first author is a member of GNSAGA of INdAM and is supported by the fund Universit\`a degli Studi di Trieste - FRA 2023.} 
\thanks{The second author has been partially supported by the grant PID2019-104844GB-I00}

\begin{abstract} The Jacobian scheme of a reduced, singular projective plane curve is the zero-dimensional scheme, whose homogeneous ideal is generated by the partials of its defining polynomial. The degree of such a scheme is called the global Tjurina number and, if the curve is not a set of concurrent lines, some upper and lower bounds depending on the degree of the curve and the minimal degree of a Jacobian syzygy, have been given by
A.A. du Plessis and C.T.C. Wall.

In this paper we give a complete 
geometric characterization of conic-line arrangenents, with global Tjurina number attaining the upper bound. Furthermore, we characterize conic-line arrangenents attaining the lower bound for the global Tjurina number, among all curves with a linear Jacobian syzygy.

As an application, we characterize conic-line arrangenents with Jacobian scheme equal to an eigenscheme of some ternary tensor, and we study the geometry of their polar maps.
\end{abstract}

\maketitle

\section{Introduction} 
The Jacobian scheme $\Sigma_f$ of a reduced, singular projective plane curve $C=V(f) \subseteq \PP^2$ is the zero-dimensional scheme, whose homogeneous ideal is generated by the partials of $f$. The degree of such a scheme is called the {\it global Tjurina number} $\tau(C)$ and it is equal to $(d-1)^2$, if $C$ consists of concurrent lines, while in all other cases, a theorem by
A.A. du Plessis and C.T.C. Wall 
in \cite{PW} determines the following bounds on $\tau(C)$ in terms of the minimal degree $r$ of a syzygy between the three partials:
$$
 (d-1)(d-r-1)\le \tau(C)\le (d-1)(d-r-1)+r^2.
 $$
 In particular, we have $\tau(C) \le (d-1)(d-2) +1$.

The scheme structure of Jacobian schemes is, in general, not completely understood, even in the case of irreducible curves. For instance, a class of curves attaining the bound above for any $r\ge 2$ is given by some rational cuspidal curves, as shown in \cite [Theorem 1.1] {DS4}. 
The set of curves attaining the maximal bound or one less seems to be very rich, and examples 
having high genus and many branches have been given in \cite[Theorem 3.9 and Theorem 3.11]{BGLM}. 
As far as we know, the only characterization result is given in the case of a linear Jacobian syzygy by 
A.A. du Plessis and C.T.C. Wall 
in \cite[Proposition 1.1]{PW2}, which states that $r=1$ and only if the curve admits a $1$-dimensional symmetry, i.e. the curve admits a $1$-dimensional algebraic subgroup of $PGL_2(\CC)$ as automorphism group.

In this paper we focus on the case of conic-line arrangements with a linear Jacobian syzygy. Some results concerning conic-line arrangements with only nodes, tacnodes, and ordinary triple points are given in \cite{DP}, but in degree $d \ge 5$ one has $r \ge 2$. The Jacobian syzygies and the global Tjurina number of conics arrangements of conics belonging to particular pencils can be found in
\cite{Pl}, \cite{Sh} and \cite{D2}.

The assumption $r=1$ implies
$(d-1)(d-2)\le \tau(C)\le (d-1)(d-2)+1$, and any curve attaining the upper bound is reducible by
\cite[Proposition 1.3]{PW2}. Concerning the lower bound, there are irreducible curves, with a linear syzygy, satisfying $\tau(C)=(d-1)(d-2)$, like for instance some Sebastiani - Thom rational cuspidal curves (see \cite[Proposition 2.11]{DS2}).

We give a complete 
geometric characterization of the conic-line arrangements attaining the two bounds for $r=1$.
Concerning the upper bound, examples are given by line arrangements consisting of the union of $d-1$ concurrent lines and one general line,
as described in \cite[Proposition 4.7(5)]{DIM}.
Our first main result is the following (see Theorem 
\ref{prop: maxtau}): 

\begin{thmx}
 Let $C=V(f)$ be a conic-line arrangement in $\PP^2$ of degree $d\ge 5$. Then, $\tau(C)=(d-1)(d-2)+1$ if and only if $C$ is either:
 \begin{itemize}
  \item $\sL$: a line arrangement with $d-1$ concurrent lines and a general line;
  \item $\sC_1$: a union of conics belonging to a hyperosculating pencil, that is with base locus supported in a point;
  \item $\sC \sL_1$: a union of conics belonging to a hyperosculating pencil and the tangent line in the hyperosculating point;
  \item $\sC \sL_2$: the union of conics belonging to a bitangent pencil and a tangent line in one of the bitangency points;
  \item $\sC \sL_3$: the union of conics belonging to a bitangent pencil and the two tangent lines in the bitangency points;
  \item $\sC \sL_4$: the union of conics belonging to a bitangent pencil, one tangent line in a tangency point, and the line connecting the two tangency points;
  \item $\sC \sL_5$: the union of conics belonging to a bitangent pencil, the two tangent lines in the bitangency points, and the line connecting the two tangency points.
  \end{itemize}
\end{thmx}

 In the case $\tau(C)=(d-1)(d-2)$, the characterization of conic-line arrangements is given by the following result (see Theorem \ref{taumin}):
\begin{thmx}
Let $C=V(f)$ be a reduced conic-line arrangement in $\PP^2$ of degree $d\ge 6$. Then, $\tau(C)=d^2-3d+2$ if and only if 
 $C$ is either:
 \begin{itemize}
   \item ${\mathcal C}_2$:
 a conic arrangement given by the union of conics belonging to a bitangent pencil;
 
 \item ${\mathcal C}{\mathcal L}_6$: a conic-line arrangement 
 given by the union of conics belonging to a bitangent pencil and the line passing through the two bitangency points.
\end{itemize}
\end{thmx}

As a consequence of our first result, we can determine the degree of the 
polar map of conic-line arrangements with quasihomogenous singularities, that is with total Milnor number $\mu(C)$ (see Definition \ref{def: Milnor and Tjurina}) equal to the total Tjurina number. Indeed, recall that, when $C=V(f)$ is not a set of concurrent lines, the polar map $\nabla f$ associated with $f$ defines a generically finite rational map $\nabla f : \PP^2 \dasharrow \PP^2$ of degree $(d-1)^2 - \mu(C)$. Therefore, when $\mu(C) = \tau(C)$ and $\tau(C)$ is maximal, such a degree is minimal and equal to $d-2$. This occurs in cases $\sL$ and $\sC\sL_2$. 

Another related problem
is a Torelli-type question, posed by Dolgachev and Kapranov (see 
\cite{DK}), which asks whether the rank $2$ vector bundle of logarithmic vector fields $\sT\langle C \rangle$, given as the kernel of the map 
$$ (\partial _x f, \partial_y f, \partial_z f):\sO^{\oplus 3}_{\PP^2}(1) \to \sJ_f(d),
$$ with $\sJ_f$ the sheafyfied Jacobian ideal,
determines uniquely the curve, see also \cite{DSer} and \cite {D3}. In the cases under consideration
this result does not hold. 

Finally, we observe that in the case of maximal Tjurina number and linear syzygy given by three linearly independent forms, the Jacobian schemes 
$\Sigma_f$ turn out to be also eigenschemes (for the definition see \ref{def: eigenscheme}) of suitable partially symmetric tensors of order $d-1$. This happens in the cases $\sL$ and $\sC \sL_2$,
and we can apply the results of \cite{OO} and \cite{BGV}. In particular, such schemes arise also as zeroes of a section of 
the twisted tangent bundle $\sT_{\PP^2} (d-3)$,
and we have that
the blow-up ${\rm Bl}_{\Sigma_f} \PP^2 \subset \PP^2 \times \PP^2$ is a complete intersection of the projective
bundle $\PP(\sT_{\PP^2})$ and a divisor of bidegree $(1,d-2)$, the possible contracted curves by the polar map are only lines, and as a consequence no subscheme 
of degree $k(d-1)$ is contained in a curve of degree
$k$, for any $2 \le k \le d-2$. We further specify that in the line arrangement case $\sL$, the contracted
lines are precisely the components of $V(f)$, while in the conic-line arrangement, the only contracted line is the tangent line appearing in the configuration $V(f)$.

The study of singular curves with a minimal Jacobian syzygy of higher degree seem to be very involved and wild. We believe that the approach concerning the study of the fibers of the polar map deserves further investigations.

The techniques involved in our study rely on a result relating the syzygy module of a product of polynomials, with no common factor, given in \cite[Theorem 5.1 and Corollary 5.3] {DIS}, and on the characterization of the Hilbert-Burch matrix in the case of particular linear Jacobian syzygies given in 
\cite[Theorem 3.5]{BC}. The proofs of our main results are based on a careful analysis of the possible Jacobian syzygies of any subcurve of the considered curves.

The organization of the paper is the following: 
in the next section we recall definitions and preliminary results regarding Jacobian ideals, Jacobian sygyzies and eigenschemes
of ternary tensors.

In Section 3 we determine the geometric classification of Jacobian schemes 
corresponding to Jacobian ideals with a linear syzygy.

Finally, in Section 4, we characterize 
conic-line arrangements with a Jacobian scheme, which is also an eigenscheme of some ternary tensor. For such curves we describe the 
degree $d-2$ generically finite polar map, characterizing its contracted curves.

\vskip 2mm
\noindent {\bf Acknowledgement}. Most of this work was done while the first author was a guest of the Universitat de Barcelona, and she would like to thank the people of the Departament de Matem\`atiques i Inform\`atica for their warm hospitality.

\section{Preliminaries} 
This section contains the basic definitions and results on Jacobian ideals associated to reduced singular plane curves as well as on eigenschemes, and it lays the groundwork for the results in the later sections.

From now on, we fix the polynomial ring $R=\CC[x,y,z]$ and we denote by $C =V(f)$ a reduced curve of degree $d$ in the
complex projective plane $\PP^2=\Proj(R)$ defined by a homogeneous polynomial $f\in R_d$.

\vskip 2mm

\subsection{Jacobian ideal of a reduced curve} The Jacobian ideal $J_f$ of 
 a reduced singular plane curve $C =V(f)$ of degree $d$ is definied as the homogeneous ideal in $R$ generated by the 3 partial derivatives
$\partial_x f$, $\partial_y f$ and $\partial_z f$. We denote by ${\rm Syz}(J_f)$ the graded $R$-module of all Jacobain relations for $f$, i.e.,
$$ 
{\rm Syz}(J_f):=\{(a,b,c)\in R ^3\mid a\partial_x f + b\partial_y f+ c\partial_z f = 0 \}.
$$
We will denote by ${\rm Syz}(J_f)_t$ the homogeneous part of degree $t$ of the graded $R$-module ${\rm Syz}(J_f)$;
for any $t \ge 0$, we have that ${\rm Syz}(J_f)_t$ is a $\CC$-vector space of finite dimension.
The minimal degree of a Jacobian syzygy for $f$ is the integer ${\rm mrd}(f)$ defined to be the smallest integer $r$ such that there
is a nontrivial relation
$ a\partial_x f + b\partial_y + c\partial_z = 0$
among the partial derivatives $\partial_x f$, $\partial_y f$ and $\partial_z f$ of $f$ with coefficients $a, b, c \in R_r$. More precisely, we have:
$${\rm mrd}(f)=min\{ n\in \NN \mid {\rm Syz}(J_f)_n\ne 0\}.
$$
It is well known that ${\rm mrd}(f) = 0$, i.e. the three partials
$\partial_x f$, $\partial_y f$ and $\partial_z f$ are linearly dependent, if and only if $C$ is
a union of lines passing through one point $p\in \PP^2$. So, we will always assume that ${\rm mrd}(f)>0$ and one of our goals will be to give a geometric classification of conic-line arrangements $C=V(f)$ of degree $d$ with ${\rm mrd}(f)=1$, see Theorems \ref{prop: maxtau} and \ref{taumin}.

\begin{definition} Let $C =V(f)$ be a reduced singular plane curve of degree $d$. We say that $C$ is {\em free} if the graded $R$-module ${\rm Syz}(J_f)$ of all Jacobian relations for $f$ is a free $R$-module, i.e. 
\begin{equation}\label{free}
{\rm Syz}(J_f)=R(-d_1)\oplus R(-d_2)
\end{equation}
with $d_1+d_2=d-1$. In this case $(d_1,d_2)$ are called the {\it exponents} of $C$.

We say that $C$ is {\em nearly free} if the minimal free resolution of ${\rm Syz}(J_f)$ looks like:
\begin{equation}\label{nearly free}
0 \longrightarrow R(-d-d_2) \longrightarrow R(1-d-d_1)\oplus R(1-d-d_2)^2 \longrightarrow {\rm Syz}(J_f) \longrightarrow 0
\end{equation}
with $d_1\le d_2$ and $d_1+d_2=d$.
\end{definition}

\begin{example}
(1) The rational cuspidal quintic $C\subset \PP^2$ of equation $C=V(f)=V(y^4z+x^5+x^2y^3)$ is free. Indeed, $J_f=(5x^4+2xy^3,3x^2y^2+4y^3z,y^4)\subset R$ and it has a minimal free $R$-resolution of the following type:
$$
0 \longrightarrow R(-6)^2 \longrightarrow R(-4)^3 \longrightarrow J_f \longrightarrow 0. 
$$
We have ${\rm mrd}(f)=2$, $\deg(J_f)=12$ and $C$ is free.
\vskip 2mm
\noindent{}
(2) The rational cuspidal quintic $C\subset \PP^2$ of equation $C=V(f) =V(y^4z+x^5)$ is nearly free. Indeed, $J_f=(5x^4,4y^3z,y^4)\subset R$ and it has a minimal free $R$-resolution of the following type:
$$
0 \longrightarrow R(-9) \longrightarrow R(-5)\oplus R(-8)^2 \longrightarrow R(-4)^3 \longrightarrow J_f \longrightarrow 0. 
$$
We have ${\rm mrd}(f)=1$, $\deg (J_f)=12$ and $C$ is not free but it is nearly free.

\vskip 2mm
\noindent 
(3) The nodal quintic $C\subset \PP^2$ of equation $C=V((x^2+y^2+z^2)(x^3+y^3+z^3))=0$ is neither free nor nearly free. Indeed, $J_f=(5x^4+3x^2(y^2+z^2)+2x(y^3+z^3),2x^3y+3x^2y^2+5y^4+3y^2z^2+2yz^3,2x^3z+2y^3z+3x^2z^2+3y^2z^2+5z^4)\subset R$ and it has a minimal free $R$-resolution of the following type:
$$
0 \longrightarrow R(-9)\oplus R(-10) \longrightarrow R(-7)\oplus R(-8)^3 \longrightarrow R(-4)^3 \longrightarrow J_f \longrightarrow 0. 
$$
We have ${\rm mrd}(f)=3$ and $\deg (J_f)=6$. 
\end{example}

In general, the condition that a reduced singular curve $C=V(f)$ in $\PP^2$ is free is
equivalent to the Jacobian ideal $J_f$ of $f$ being arithmetically Cohen-Macaulay
of codimension two; such ideals are completely described by the
Hilbert-Burch theorem \cite{e}: if $I = \langle g_1,\ldots,
g_m\rangle \subset R$ is a Cohen-Macaulay ideal of codimension two, then $I$ is defined
by the maximal minors of the $(m+1)\times m$ matrix of the first
syzygies of the ideal $I$. Combining this with Euler's formula for a homogeneous
polynomial, we get that a free curve $C=V(f)$ in $\PP^2$ has a very
constrained structure: $f = \det(M)$ for a $3 \times 3$ matrix $M$, with one row consisting of the 3 variables, and the
remaining $2$ rows are the minimal first syzygies of $J_f$.

\vskip 2mm
Free curves are related with the total Tjurina number. We first recall some notions from singularity theory.

Let $C=V(f)\subset \mA^2$ be a reduced, not necessarily irreducible, plane curve and fix a singular point 
$p\in C$. 
Let $\CC \{x, y \}$ denote the ring of convergent power series.

\begin{definition}\label{def: Milnor and Tjurina}
The {\em Milnor number} of a reduced plane curve $C=V(f)$ at $(0,0)\in C$ is 
$$
\mu _{(0,0)}(C) = \dim \CC \{x, y \}/ \langle \partial_x f,\partial_y f \rangle.
$$ 

The {\em Tjurina number} of a reduced plane curve $C=V(f)$ at $(0,0)\in C$ is
$$
\tau _{(0,0)}(C) = \dim \CC \{x, y \}/ \langle \partial_x f,\partial_y f ,f\rangle.
$$

To define $\mu _p(C)$ and $\tau _p(C)$ for an arbitrary point $p$, translate $p$ in the origin.
\end{definition}

We clearly have $\tau _{(0,0)}(C)\le \mu _{(0,0)}(C)$.
For a projective plane curve $C=V(f)\subset \PP^2$, it holds:
$$
\tau (C):=\deg J_f = \sum _{P\in Sing(C)} \tau _p(C)
$$
where $J_f$ is the Jacobian ideal. We call to $\tau (C)$ the {\em total Tjurina number} of $C$.

\vskip 2mm A nice result of du Plessis and Wall gives upper and lower bounds for the total Tjurina number $\tau(C)$ of a reduced plane curve $C=V(f)\subset \PP^2$ in terms of its degree $d$ and the minimal degree ${\rm mrd}(f)$ of a syzygy of its Jacobian ideal $J_f$, and relates the freeness of a curve with $\tau(C)$. More precisely, we have:

\begin{proposition}\label{bounds_tau} Let
$C=V(f)$ be a reduced singular plane curve of degree $d$ and let $r:={\rm mrd}(f)$. It holds:
$$
(d-1)(d-r-1)\le \tau(C)\le (d-1)(d-r-1)+r^2.
$$
Moreover, if $\tau(C)= (d-1)(d-r-1)+r^2$, then the curve $C$ is free, and such as condition is also sufficient if $d> 2r$.
\end{proposition}
\begin{proof}
    See \cite[Theorem 3.2]{PW} and \cite[Corollary 1.2] {D}.
\end{proof}

Next result will play an important role in next section. It relates the minimal degree ${\rm mrd}(f)$ of a syzygy of the Jacobian ideal of a reducible plane curve $C=C_1\cup C_2=V(f_1f_2)$ with the minimal degrees ${\rm mrd}(f_1)$ and ${\rm mrd}(f_2)$ of a Jacobian syzygy of $C_1=V(f_1)$ and $C_2=V(f_2)$, respectively. Observe that the syzygy module can be identified with the module of derivations {\it killing} the polynomial $g$, that is the submodule $D_0 (g)$ of the free $R$-module 
$D(g) =\{ \delta = a \partial_x +b \partial _y +c \partial_z \ : \ a,b,c \in R\}$ of $\CC$-derivations of the polynomial ring $R$ annihilated by $g$:
$$
D_0 (g)= \{ \delta \in D(g) \ : \ \delta g =0\}.
$$
In the case of a smooth curve $V(g)$, the syzygy module is trivial, so we will indeed consider the module 
$D_0 (g)$ instead of ${\rm Syz}(J_g)$.

\begin{theorem}\label{r_union} Let $C_i=V(f_i)$ for $i=1,2$ be two reduced curves in $\PP^2$ without common irreducible components. Set $d_i=\deg f_i$ and $r_i={\rm mrd}(f_i)$ for $i=1,2$. Let $C=V(f_1f_2)$ be the union of $C_1$ and $C_2$, let $d=d_1+d_2=\deg f$ and $r={\rm mrd}(f)$. Then it holds:

\begin{enumerate}
\item
If $\delta_1 
\in D_0 (f_1)$, then
$$
\delta= f_2 \delta_1 - \frac{1}{d} \delta_1 ( f_2)
\ E \in D_0 (f),
$$
where 
$E=x\partial_x +y \partial_y+z\partial_z$ denotes the Euler
derivation;

\item  $D_0 (f) \subset D_0 (f_1) \cap D_0 (f_2)$; more precisely, for $\delta \neq 0$, one has 
$\delta \in D_0 (f)$ if and only if $\delta$ can be written in a unique way in the form $\delta = h\ E + \delta_1=-h\ E + \delta_2$, where $h$ is a suitable homogeneous polynomial and $\delta_j \in D_0 (f_j)$ are non-zero for $j=1,2$.

\item In particular, we have
$$
\max(r_1,r_2)\le r\le \min(r_1+d_2,r_2+d_1).
$$
and $r$ is the minimal integer $t$ such that either $D_0 (f_1)_t\cap D_0 (f_2)_t\ne 0$ or $D_0 (f_1)_t + D_0 (f_2)_t$ contains a non-zero multiple
of the Euler derivation $E$.
\end{enumerate}
\end{theorem}
\begin{proof}
See \cite[Theorem 5.1]{DIS} and \cite[Corollary 5.3]{DIS}.
\end{proof}

\subsection{Eigenschemes in $\PP^2$}  
Since we shall investigate whether a Jacobian scheme is also an eigenscheme of some tensor, we conclude the preliminary section with the basic definitions and results concerning tensor eigenschemes.

There are several notions of eigenvectors and eigenvalues for tensors, as introduced independently in \cite{L} and \cite{Q}. Here we focus our attention on the algebraic-geometric point of view. We choose a basis for $\CC^3$, we identify a partially symmetric tensor $T$ with a triple of homogeneous polynomials of degree $d-2$ and we describe the eigenpoint of a tensor $T$ algebraically by the vanishing of the minors of a homogeneous matrix. More precisely, we have:

\begin{definition}  \label{def: eigenscheme}
Let $T=(g_1,g_2, g_3)\in (Sym^{d-2}\CC^{3})^{\oplus 3}$ be a partially symmetric tensor. The {\em eigenscheme} of $T$ is the closed subscheme $E(T)\subset \PP^2$ defined by the $2\times 2$ minors of the homogeneous matrix
\begin{equation}\label{def_matrix}
M=\begin{pmatrix} x & y  & z \\
g_1 & g_2  & g_3
\end{pmatrix}.
\end{equation}
 If $T$ is general, then $E(T)$ is a $0$-dimensional scheme (see, for instance, \cite{Abo}).
Moreover, by the Hochster-Eagon Theorem \cite{HE}, the coordinate ring $R/I(E(T))$ is a Cohen-Macaulay ring, and as a consequence, the homogeneous ideal $I(E(T))$ is saturated. Hence $E(T)$ is a standard determinantal scheme.
When the tensor $T$ is symmetric, i.e., there is a homogeneous polynomial $f$ and $g_0=\partial_x f$, $g_1=\partial_y f$ and $g_2=\partial_z f$, we denote its eigenscheme by $E(f)$.
\end{definition}
 It is worthwhile to point out that in the case of a symmetric tensor corresponding to some homogeneous polynomial $f$, the eigenpoints are the fixed points of the polar map $$\nabla f=(\partial_xf,\partial_yf,\partial _zf):\PP^2 \dasharrow \PP^2$$ of $f$.
 \begin{example}
\label{ex: Fermat}
\rm
We consider the Fermat cubic $V(f)$ with $f=x^3+y^3+z^3\in \CC [x,y,z]$. The eigenscheme $E(f)$ is the 0-dimensional subscheme of $\PP^2$ of length 7 defined by the maximal minors of 
$$
M=\begin{pmatrix} x & y & z \\
x^2 & y^2 & z^2 
\end{pmatrix}.
$$
Therefore, $E(f)=\{(1,0,0,),(0,1,0),(0,0,1),(1,1,0),(1,0,1),  (0,1,1),(1,1,1)\}.$
\end{example}

 If we fix an integer $d\ge 2$ and 
 $T=(g_1,g_2,g_3)\in (Sym^{d-2}\CC^{3})^{\oplus 3}$ is a general partially symmetric tensor, then it holds (see \cite[Theorem A2.10]{e}):
 \begin{itemize}
     \item[(1)] $E(T)$ is a reduced 0-dimensional scheme of length 
          $d^2-3d+3$.
     \item[(2)] The homogeneous ideal $I(E(T))\subset R$ has a minimal free $R$-resolution
     $$ 0 \longrightarrow R(-2d+3)\oplus R(-d) \longrightarrow   
      R(-d+1)^{3}\longrightarrow I(E(T)) \longrightarrow 0.
     $$
     \end{itemize}
The two conditions above are not sufficient for a planar $0$-dimensional subscheme to be an eigenscheme, and a characterization is given by the following result (see \cite[Proposition 5.2]{ASS}.

\begin{proposition}
Let Z be a $0$-dimensional subscheme of $\PP^2$ of degree $d^2-3d+3$. Then $Z$ is the eigenscheme of a tensor if and only if its Hilbert-Burch matrix has the form
$$
\begin{pmatrix} L_1 & G_1 \\
L_2 & G_2\\
L_3 & G_3\\
\end{pmatrix},
$$
where $L_1,L_2,L_3$ are linearly independent linear forms.
\end{proposition}


\section{
Conic-line arrangements with a linear Jacobian syzygy}

We start this section with two series of examples of reduced conic-line arrangements. All these examples will play an important role since, as we will see, they are the only examples of reduced conic-line arrangements $C$ in $\PP^2$, whose Jacobian ideal has a linear syzygy.

In what follows we shall use the result \cite[Theorem 3.5]{BC},
due to R. O. Buchweitz and A. Conca,
which determines the Hilbert - Burch matrix of Jacobian schemes admitting a linear syzygy of the type $(ax,by,cz)$ for some coefficients $a,b,c \in \CC$. For completeness, we recall its statement:

\begin{theorem}[Buchweitz - Conca]\label{Conca}
  Let $K$ be a field of characteristic zero and 
  $f \in K[x,y,z]$ a reduced polynomial of degree $d$ in three variables such that $f$ is contained in the ideal of its partial derivatives.
Assume further that there is a triple $(a,b,c)$ of elements of $K$ that are not all zero such that 
$ax\ \partial_x f +by\ \partial _y f + cz\ \partial_z f=0$.

We then have the following possibilities, up to renaming the variables: 
\begin{enumerate}
\item If $abc \neq 0$, then $f$ is a free divisor with Hilbert-Burch matrix
\begin{equation}\label{eq: Hilbert Burch}
\begin{pmatrix} ax & \qquad \left(\frac{1}{c} - \frac{1}{b}\right)\ (d+2)^{-1} \partial_{yz}f \\
by & \qquad \left(\frac{1}{a} - \frac{1}{c}\right)\ (d+2)^{-1}\partial_{xz}f\\
cz & \qquad \left(\frac{1}{b} - \frac{1}{a}\right)\ (d+2)^{-1}\partial_{xy}f\\
\end{pmatrix},
\end{equation}
where $\partial_{**}f$ denotes the corresponding second order derivative of $f$.

\item If $a = 0$, but $bc \neq 0$, then $f$ is a free divisor if, and only if, $\partial_x f\in (y,z)$. If that condition is verified and $\partial_x f = yg + zh$, then $\frac {\partial f_y} {cz} = \frac {- \partial_z f}{by}$ is an element of $K[x,y,z]$, and a Hilbert-Burch matrix is given by
 \begin{equation}
\begin{pmatrix} 0 &  {\partial_y f}/ {cz}  \\
by & {-h}/{c} \\
cz & {g}/{b} \\
\end{pmatrix}.
\end{equation}
\item If $a = b = 0$, then $f$ is independent of $z$ and, so, being the suspension of a reduced plane curve, is a free divisor.

\end{enumerate}
\end{theorem}

We focus now on conic-line arrangements.
The first series of examples corresponds to reduced conic-line arrangements $C=V(f)$ of degree $d$ with ${\rm mrd}(f)=1$ and maximal Tjurina number $\tau(C)=(d-1)(d-2)+1=d^2-3d+3$.

\begin{example}\label{tau_max} 
\begin{enumerate}
    \item We fix an integer $d\ge 3$. Let ${\mathcal L}$ be a line arrangement with $d-1$ lines through a point $p$, and one other line in general position. Without loss of generality we can assume that $p=(0:0:1)$ and that the general line is $V(x)$, so that the equation of the line arrangement ${\mathcal L}$ is given by
    \begin{multicols}{2}
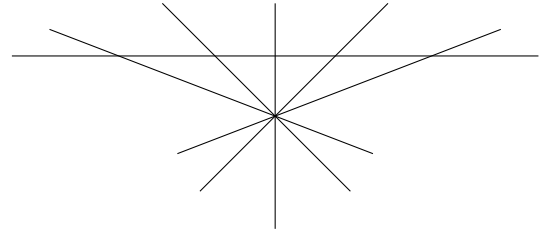
$$
{\mathcal L}: \ z\prod _{i=1}^{d-1}(a_ix+b_iy)=0
$$
\bigskip 
\begin{center}
\begin{tikzpicture}
 \draw (0,-1.5) -- (0,1.5);
 \draw (-1,-1) -- (1.5,1.5);
 \draw (1,-1) -- (-1.5,1.5);
 \draw (-1.3,-0.5) -- (3,1.154);
 \draw (1.3,-0.5) -- (-3,1.154);
 \draw (-3.5,0.8) -- (3.5,0.8);
\end{tikzpicture}
\end{center}
\end{multicols}
\noindent with $(a_i:b_i)\ne (a_j:b_j)$ for $i\ne j$. 
It is simple to determine a linear syzygy between the three partials of $f$, by observing that $\partial_z f=\prod _{i=1}^{d-1}(a_ix+b_iy)$. So, we have
$$
f= z \ \partial_z f.
$$
On the other hand, by Euler formula we also have $ f=\frac {1}{d} (x \ \partial_x f +y \partial_y f +z\ \partial_x f)$, hence we get the identity
$$
x\ \partial_x f +y \ \partial_y f +(1-d)z\ \partial_x f =0.
$$
Therefore, according to Theorem \ref{Conca}, (1), the Hilbert-Burch matrix of
$J_f$ is given by
\begin{equation}\label{eq: Hilbert Burch L}
\begin{pmatrix} x & \frac{1}{(1-d)} \ \partial_{yz}f \\
y & \frac {1}{(d-1)}\partial_{xz}f\\
(1-d)z & 0\\
\end{pmatrix},
\end{equation}
a minimal free $R$-resolution of $J_f$ is given by:
 $$
0 \longrightarrow R(-d)\oplus R(-2d+3) \longrightarrow R(-d+1)^3 \longrightarrow J_f \longrightarrow 0, 
$$
and $\sL$ is free with exponents $(1,d-2)$
and global Tjurina number $d^2-3d+3$.

It is worthwhile to point out that the 3 entries 
$(x,y,(1-d)z)\in {\rm Syz}(J_f)_1$ of the linear syzygy are linearly independent.

\item We fix an even integer $d=2m\ge 4$. Let ${\mathcal C}_1$ be a conic arrangement with $m$ conics $C_1,\dots , C_m$ such that there exists a point $p\in \PP^2$, for all $i,j$, $1\le i<j\le m$, satisfying $C_i\cap C_j=\{p\}$, and the intersection point $p$ is a singularity $A_7$ for $C_i\cup C_j$. In other words, the $m$ conics belong to a hyperosculating pencil; such a curve is called an {\it even P\l oski curve} in \cite{Sh}. Without loss of generality, we can assume that $p=(0:0:1)$ and the equation of the conic arrangement ${\mathcal C}_1$ is given by
\begin{multicols}{2}
$$
{\mathcal C}_1: \ f=\prod _{i=1}^m(x^2 + a_i (xz + y^2)) = 0, 
$$
\bigskip 
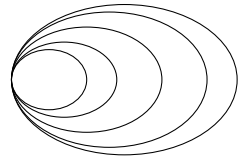
\begin{center}
\begin{tikzpicture}
 \draw (2,0) ellipse (1.5 and 1);
 \draw (1.8,0) ellipse (1.3 and 0.9);
 \draw (1.5,0) ellipse (1 and 0.7);
 \draw (1.2,0) ellipse (0.7 and 0.5);
 \draw (1,0) ellipse (0.5 and 0.4);
\end{tikzpicture}
\end{center}
\end{multicols}

\noindent with $a_i\ne 0$ and $a_i\ne a_j$ for $i\ne j$. The reduced plane curve 
$\sC_1$ has degree $d=2m$. A linear Jacobian syzygy can be determined by observing that $(0,x,-2y)$ is a linear syzygy of the Jacobian ideal of $f_i:=x^2 + a_i (xz + y^2)$, for any $i=1,\dots ,m$, so we deduce that ${\rm Syz}(J_f)_1$ is also generated by $(0,x,-2y)$.

Moreover, we claim that 
$\sC_1$ is free with exponents $(1,d-2)$, so that the global Tjurina number is $d^2 - 3d+3$.

To prove the claim, observe that since $r=1$, the conic arrangement $\sC_1$ is either free or nearly free, and by \eqref{free} and \eqref{nearly free}, it is free if and only if $J_f$ admits a sygyzy of degree $d-2$, which is not proportional to $(0,x,-2y)$.
Let us prove the latter fact by induction on the number $m$ of conics.

If $m=2$, a degree $2$ syzygy, which is not proportional to $(0,x,-2y)$, is given by
$$
\qquad \qquad
(-(a_1 +a_2)x^2-2a_1 a_2(y^2 + xz), a_1 a_2 yz,
4x^2 +2(a_1 +a_2)y^2 +3(a_1+a_2) x z +2a_1 a_2 z^2).
$$
Now assume that $m \ge 3$, and that any conic arrangement of $m-1$ conics belonging to a hyperosculating pencil admits a syzygy of degree $2m-4$, not proportional to $(0,x,-2y)$.
Let $f=\prod _{i=1}^m(x^2 + a_i (xz + y^2))$, and set 
$$
f_1 = \prod _{i=1}^{m-1}(x^2 + a_i (xz + y^2)), \qquad
f_2 = x^2 + a_m (xz + y^2).
$$
By induction hypothesis
$J_{f_1}$ admits a syzygy $\delta_1 \in {\rm Syz}(J_{f_1})_{2m-4}$, with $\delta_1 \not \in \langle
(0,x,-2y)\rangle$, and by Theorem \ref{r_union}, (1), 
we have
$$
\delta= f_2 \delta_1 - \frac{1}{2m} (\delta_1 \cdot 
\nabla f_2)
\ E \in {\rm Syz}(J_{f})_{2m-2},
$$
where we set $\delta_1 \cdot \nabla f_2= h_1 \partial_x f_2+ h_2 \partial_y f_2+ h_3 \partial_z f_2$, and $E=(x,y,z)$ is the Euler relation.
As observed in the proof of such a Theorem (see \cite[Theorem 5.1]{DIS}), since $\delta_1 \neq 0$, it is also $\delta \neq 0$. Finally, we claim that $\delta \not \in \langle
(0,x,-2y)\rangle$. Indeed, if $\delta_1 \cdot 
\nabla f_2=0$, we have $\delta= f_2 \delta_1$ and since
$\delta_1 \not \in \langle
(0,x,-2y)\rangle$ by induction hypothesis, the claim follows. If $\delta_1 \cdot 
\nabla f_2 \neq 0$, we see that 
$$
\delta \cdot \nabla f_1 = f_2 \delta_1 \cdot \nabla f_1- \frac{1}{2m} (\delta_1 \cdot 
\nabla f_2)
\ E \cdot \nabla f_1 = - \frac{2m-2}{2m}  (\delta_1 \cdot 
\nabla f_2) f_1\neq 0.
$$
On the other hand, if we had $\delta = h\ (0,x,-2y)$
for some polynomial $h$,
we would have
$$
\delta \cdot \nabla f_1 =h\ (0,x,-2y) \cdot \nabla f_1=0,
$$
as $(0,x,-2y) \in {\rm Syz} (J_{f_1})$.

 It is important to point out that in this case the linear syzygy of $J_f$ has only 2 linearly independent entries, and that such an example is not of the type considered in Buchweitz - Conca Theorem \ref{Conca}. 

\item  We fix an odd integer $d=2m+1\ge 5$. Let ${\mathcal C}{\mathcal L}_1$ be a conic-line arrangement with $m$ conics $C_1,\dots , C_m$ and a line $\ell $ such that there exists a point $p\in \PP^2$, $\ell $ is a common tangent line to all $C_i$'s, $C_i\cap C_j=\{p\}$ and the intersection point $p$ is a singularity $A_7$ for $C_i\cup C_j$, for all $i,j$, $1\le i<j\le m$. In other words, the conics belong to a hypersculating pencil; such a curve is called an {\it odd P\l oski curve} in \cite{Sh}.
 Without loss of generality we can assume that $p=(0:0:1)$, the line
 $\ell=V(x)$ and the equation of the conic-line arrangement ${\mathcal C}{\mathcal L}_1$ is given by
\begin{multicols}{2}
$$
{\mathcal C}{\mathcal L}_1: \ f=x\prod _{i=1}^m(x^2 + a_i (xz + y^2)) = 0
$$
\bigskip
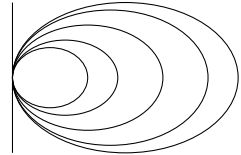
\begin{center}
\begin{tikzpicture}
 \draw (2,0) ellipse (1.5 and 1);
 \draw (1.8,0) ellipse (1.3 and 0.9);
 \draw (1.5,0) ellipse (1 and 0.7);
 \draw (1.2,0) ellipse (0.7 and 0.5);
 \draw (1,0) ellipse (0.5 and 0.4);
 \draw (0.5,-1) -- (0.5,1);
\end{tikzpicture}
\end{center}
\end{multicols}
\noindent with $a_i\ne 0$ and $a_i\ne a_j$ for $i\ne j$, and $a_i \neq a_j$ if $1\neq j$. The reduced plane curve $\sC \sL_1$ has degree $d=2m+1$; by using the same argument as in the previous example, it is not difficult to see that $\sC \sL_1$ is free with exponents $(1,d-2)$
 and global Tjurina number $d^2-3d+3$. The linear syzygy of ${\rm Syz}(J_f)_1$ is generated by $(0,x,-2y)$. So, again the linear syzygy of the Jacobian ideal of $f$ has only two independent entries, and is not of the type considered in Buchweitz - Conca Theorem \ref{Conca}.

\item  We fix an odd integer $d=2m+1\ge 5$. Let ${\mathcal C}{\mathcal L}_2$ be a conic-line arrangement with $m$ conics $C_1,\dots , C_m$ and a line $\ell $ such that there exist two points $p, q\in \PP^2$ such that $C_i\cap C_j=\{p, q\}$ and the two intersection points $p, q$ are tacnodes for $C_i\cup C_j$, for all $i,j$, $1\le i<j\le m$, and $\ell $ is a common tangent line to all $C_i$'s at $p$. Without loss of generality we can assume that $p=(0:0:1)$, $q=(1:0:0)$, $\ell=V(x)$, so that the equation of the conic-line arrangement ${\mathcal C}{\mathcal L}_2$ is given by
\begin{multicols}{2}
$$
{\mathcal C}{\mathcal L}_2: 
\ f=x\prod _{i=1}^m(xz + a_iy^2) = 0
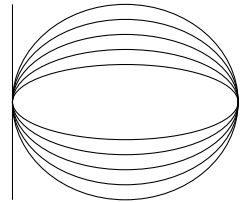
$$
\bigskip 
\begin{center}
\begin{tikzpicture}
 \draw (2,0) ellipse (1.5 and 1.3);
 \draw (2,0) ellipse (1.5 and 1.1);
 \draw (2,0) ellipse (1.5 and 0.9);
 \draw (2,0) ellipse (1.5 and 0.7);
 \draw (2,0) ellipse (1.5 and 0.5);
 \draw (0.5,-1.3) -- (0.5,1.3);
\end{tikzpicture}
\end{center}
\end{multicols}
\noindent with $a_i\ne 0$ for all $i$, $1\le i \le m$, and $a_i \neq a_j$ if $1\neq j$. The reduced singular plane curve $\sC \sL_2$ has degree $d=2m+1$. We claim that $Syz(J_f)_1=\langle ((d-1)x, -y, -(d+1)z)\rangle$.
 Indeed, set $q_i(x,y,z)= xz + a_iy^2$. We have
 $$
 \partial_x f = \prod _{i=1}^m q_i + xz \left(  \sum _{j=1}^m \prod _{i=1, i\neq j}^m  q_i\right), \ \partial_y f=2xy  \left(  \sum _{j=1}^m a_j\prod _{i=1, i\neq j}^m  q_i\right),        \ \partial_z f=
 x^2 \left(  \sum _{j=1}^m \prod _{i=1, i\neq j}^m  q_i\right),
 $$
 which, in particular, gives 
 $$
 f=x \ \partial_x f - z \ \partial_z f.
 $$
 Thus by the Euler identity we get
 $$
 (d-1)x \ \partial_x f - y \ \partial_y f -(d+1)z \ \partial_z f=0.
 $$
 Therefore, the linear syzygy of the Jacobian ideal of $f$ has again 3 linear independent entries. 
 
Therefore, according to Theorem \ref{Conca}, (1), the curve $\sC \sL_2$ is free with global Tjurina number $d^2-3d+3$, and the Hilbert-Burch matrix of $J_f$ is 
 \begin{equation}\label{eq: Hilbert Burch of $CL$}
\begin{pmatrix} (d-1)x & \frac{1}{(d+1)} \ \partial_{yz}f \\
-y & \frac {2}{(d^2-1)}\partial_{xz}f\\
-(d+1)z & -\frac {1}{(d-1)} \partial_{xy} f\\
\end{pmatrix} .
 \end{equation}

 \item The following three examples are of the type given in \cite[Example 3.7]{BC}, and they all correspond to free curves of exponents $(1,d-2)$ and ${\rm Syz}(J_f)_1$ generated by $(x,0,-z)$.
 
 \begin{itemize}
 \item Let $d=2m+2\ge 6$ and let ${\mathcal C}{\mathcal L}_3$ be a conic-line arrangement
 of degree $d$ with $m$ conics $C_1,\dots, C_m$ and two lines $\ell _1$ and $\ell _2$ such that there exist two points $p, q\in \PP^2$ satisfying $C_i\cap C_j=\{p, q\}$ and the two intersection points $p, q$ are tacnodes for $C_i\cup C_j$, for all $i,j$, $1\le i<j\le m$, $\ell _1$ is a common tangent line to all $C_i$'s at $p$, $\ell _2$ is a common tangent line to all $C_i$'s at $q$. Without loss of generality we can assume that $p=(0:0:1)$, $q=(1:0:0)$, $\ell_1=V(x)$, $\ell_2=V(z)$, so that the equation of the conic-line arrangement ${\mathcal C}{\mathcal L}_3$ is given by
\begin{multicols}{2}
$$
{\mathcal C}{\mathcal L}_3: \ f=xz\prod _{i=1}^m(xz + a_iy^2)) = 0
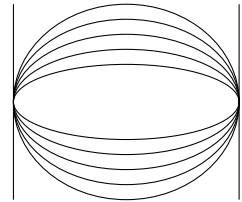
$$
\bigskip
\begin{center}
\begin{tikzpicture}
 \draw (2,0) ellipse (1.5 and 1.3);
 \draw (2,0) ellipse (1.5 and 1.1);
 \draw (2,0) ellipse (1.5 and 0.9);
 \draw (2,0) ellipse (1.5 and 0.7);
 \draw (2,0) ellipse (1.5 and 0.5);
 \draw (0.5,-1.3) -- (0.5,1.3);
 \draw (3.5,-1.3) -- (3.5,1.3);
\end{tikzpicture}
\end{center}
\end{multicols}
\noindent with $a_i\ne 0$ for all $i$, $1\le i \le m$, and $a_i \neq a_j$ if $1\neq j$.

\item Let $d=2m+2\ge 6$ and let ${\mathcal C}{\mathcal L}_4$ be a conic-line arrangement of degree $d$ with $m$ conics $C_1,\dots, C_m$ and two lines $\ell _1$ and $\ell _2$ such that there exist two points $p, q\in \PP^2$ satisfying $C_i\cap C_j=\{p, q\}$ and the two intersection points $p, q$ are tacnodes for $C_i\cup C_j$, for all $i,j$, $1\le i<j\le m$, $\ell _1$ is a common tangent line to all $C_i$'s at $p$, $\ell _2$ is the line joining $p$ and $q$. Without loss of generality we can assume that $p=(0:0:1)$, $q=(1:0:0)$, $\ell_1=V(x)$, $\ell_2=V(y)$, so that the equation of the conic-line arrangement ${\mathcal C}{\mathcal L}_4$ is given by
\begin{multicols}{2}
$$
{\mathcal C}{\mathcal L}_4: \ f=xy\prod _{i=1}^m(xz + a_iy^2)) = 0
$$
\bigskip
\begin{center}
\begin{tikzpicture}
 \draw (2,0) ellipse (1.5 and 1.3);
 \draw (2,0) ellipse (1.5 and 1.1);
 \draw (2,0) ellipse (1.5 and 0.9);
 \draw (2,0) ellipse (1.5 and 0.7);
 \draw (2,0) ellipse (1.5 and 0.5);
 \draw (0.5,-1.3) -- (0.5,1.3);
 \draw (0,0) -- (4,0);
\end{tikzpicture}
\end{center}
\end{multicols}
\noindent with $a_i\ne 0$ for all $i$, $1\le i \le m$, and $a_i \neq a_j$ if $1\neq j$.
  
 \item Let $d=2m+3\ge 6$ and let ${\mathcal C}{\mathcal L}_5$ be a conic-line arrangement 
 of degree $d$ with $m$ conics $C_1,\dots, C_m$ and three lines $\ell _1$, $\ell_2$ and $\ell_3$ such that there exist two points $p, q\in \PP^2$ satisfying $C_i\cap C_j=\{p, q\}$ and the two intersection points $p, q$ are tacnodes for $C_i\cup C_j$, for all $i,j$, $1\le i<j\le m$, $\ell _1$ is a common tangent line to all $C_i$'s at $p$, $\ell _2$ is a common tangent line to all $C_i$'s at $q$, $ell_3$ is the line joining $p$ and $q$. Without loss of generality we can assume that $p=(0:0:1)$, $q=(1:0:0)$, $\ell_1=V(x)$, $\ell_2=V(z)$ and $\ell_3 = V(y)$, so that the equation of the conic-line arrangement ${\mathcal C}{\mathcal L}_5$ is given by
\begin{multicols}{2}
$$
{\mathcal C}{\mathcal L}_5: \ f=xyz\prod _{i=1}^m(xz + a_iy^2)) = 0
$$
\bigskip
\begin{center}
\begin{tikzpicture}
 \draw (2,0) ellipse (1.5 and 1.3);
 \draw (2,0) ellipse (1.5 and 1.1);
 \draw (2,0) ellipse (1.5 and 0.9);
 \draw (2,0) ellipse (1.5 and 0.7);
 \draw (2,0) ellipse (1.5 and 0.5);
 \draw (0.5,-1.3) -- (0.5,1.3);
 \draw (3.5,-1.3) -- (3.5,1.3);
 \draw (0,0) -- (4,0);
\end{tikzpicture}
\end{center}
\end{multicols}
\noindent with $a_i\ne 0$ for all $i$, $1\le i \le m$, and $a_i \neq a_j$ if $1\neq j$.

\end{itemize}
 \end{enumerate}
 \end{example}

\bigskip
Next series of example corresponds to reduced nearly free plane curves $C=V(f)$ of degree $d$ with ${\rm mrd}(f)=1$ and $\tau(C)=d^2-3d+2$.

\begin{example}\label{mintau}
By \cite[Example 3.7]{BC}, next examples have the property that ${\rm Syz}(J_f)_1$ is generated by $(x,0,-z)$ and they are not free.
Since $r=1$, we have $\tau(C)=(d-1)(d-2)$ and they are nearly free by Lemma \ref{reduciblecomponents}.
\begin{enumerate} 
\item  Let ${\mathcal C}_2$ be a conic arrangement with $m$ conics 
 $C_1,\dots, C_m$ such that there exist two points $p, q\in \PP^2$ satisfying $C_i\cap C_j=\{p, q\}$ and the two intersection points $p, q$ are tacnodes for $C_i\cup C_j$, for all $i,j$, $1\le i<j\le m$. 
 
 Without loss of generality we can assume that $p=(0:0:1)$, $q=(1:0:0)$, so that the equation of the conic arrangement is
 \begin{multicols}{2}
$$
{\mathcal C}_2: 
\ f=\prod _{i=1}^m(xz + a_iy^2) = 0
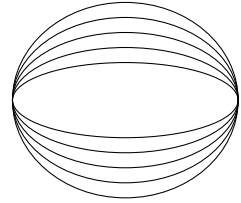
$$
\bigskip
\begin{center}
\begin{tikzpicture}
 \draw (2,0) ellipse (1.5 and 1.3);
 \draw (2,0) ellipse (1.5 and 1.1);
 \draw (2,0) ellipse (1.5 and 0.9);
 \draw (2,0) ellipse (1.5 and 0.7);
 \draw (2,0) ellipse (1.5 and 0.5);
\end{tikzpicture}
\end{center}
\end{multicols}
\noindent with $a_i\ne 0$ for all $i$, $1\le i \le m$, and $a_i \neq a_j$ if $1\neq j$.


\item Let ${\mathcal C}{\mathcal L}_6$ be a conic-line arrangement with $m$ conics $C_1,\dots, C_m$ and a line $\ell$ such that there exist two points $p, q\in \PP^2$ such that $C_i\cap C_j=\{p, q\}$ and the two intersection points $p, q$ are tacnodes for $C_i\cup C_j$, for all $i,j$, $1\le i<j\le m$, and $\ell $ is the line through $p$ and $q$.
 We can assume that $p=(0:0:1)$, $q=(1:0:0)$, $\ell =V(y)$ and the equation of the conic-line arrangement is
 \begin{multicols}{2}
$$
{\mathcal C}{\mathcal L}_6: 
\ f=y\prod _{i=1}^m(xz + a_iy^2) = 0
$$
\bigskip
\begin{center}
\begin{tikzpicture}
 \draw (2,0) ellipse (1.5 and 1.3);
 \draw (2,0) ellipse (1.5 and 1.1);
 \draw (2,0) ellipse (1.5 and 0.9);
 \draw (2,0) ellipse (1.5 and 0.7);
 \draw (2,0) ellipse (1.5 and 0.5);
 \draw (-0.5,0) -- (4.5,0);
\end{tikzpicture}
\end{center}
\end{multicols}
\noindent with $a_i\ne 0$ for all $i$, $1\le i \le m$, and $a_i \neq a_j$ if $i \neq j$.

 \end{enumerate}
\end{example}

Our next goal is to establish a geometric characterization of all reduced conic-line arrangements $C=V(f)$ in $\PP^2$ of degree $d\ge 3$, whose Jacobian ideal $J_f$ has a linear syzygy, i.e., ${\rm mrd}(f)=1$.

\begin{lemma}\label{reduciblecomponents}
Let $C=V(f)$ be a reduced singular curve in $\PP^2$ of degree $d\ge 3$ and let $J_f=(f_x,f_y,f_z)$ be its Jacobian ideal. Assume that ${\rm mrd}(f)=1$. Then
$d^2-3d+2 \le \tau(C)\le d^2-3d+3$. Moreover, if $\tau(C)= d^2-3d+3$ (resp. $\tau(C)= d^2-3d+2$) then $C$ is free (nearly free).
\end{lemma}
\begin{proof} By Proposition \ref{bounds_tau} we have $d^2-3d+2 \le \tau(C)\le d^2-3d+3$. By \cite[Theorem 1.2]{D} (see also \cite[Proposition 26]{ellia}), if $\tau(C)= d^2-3d+3$ (respectively $d^2-3d+2$) then $C$ is free (nearly free).
\end{proof}

\subsection{Free conic-line arrangements with
a linear Jacobian syzygy}

Our first goal is to classify free conic-line arrangements $C=V(f)$ of degree $d$ in $\PP^2$ with $r={\rm mrd}(f)=1$; such curves have maximum Tjurina number $\tau (C)=d^2-3d+3$. We have

\begin{theorem}\label{prop: maxtau}
Let $C=V(f)$ be a conic-line arrangement in $\PP^2$ of degree $d\ge 5$. Then, $\tau(C)=d^2-3d+3$ if and only if $C$ is either a line arrangement ${\mathcal L}$ as in example \ref{tau_max}(1), or a conic arrangement ${\mathcal C}_1$ as in example \ref{tau_max}(2), or a conic-line arrangement ${\mathcal C}{\mathcal L}_1$, ${\mathcal C}{\mathcal L}_2$, ${\mathcal C}{\mathcal L}_3$, ${\mathcal C}{\mathcal L}_4$ or ${\mathcal C}{\mathcal L}_5$ as in examples \ref{tau_max}(3)-(7).
\end{theorem}
\begin{proof} All conic-line arrangements $C\subset  \PP^2$ described in examples \ref{tau_max}(1)-(7) have total Tjurina number $\tau(C)=d^2-3d+3$ and ${\rm mdr}(f)=1$. Let us prove the converse. The hypothesis $\tau(C)=d^2-3d+3$, $d\ge 5$ and Proposition \ref{bounds_tau} imply that ${\rm mrd}(f)=1$. We distinguish several cases:

\vskip 2mm
\noindent \underline{Case 1:} $C$ is a line arrangement. By \cite[Proposition 4.7(5)]{DIM}, $C$ is the union of $d-1$ lines
through a point $p$, and one other line in general position. 

\vskip 2mm
\noindent \underline{Case 2:} Let $C=\cup _{i=1}^mC_i: \ f=\prod_{i=1}^mf_i=0$ be a conic arrangement. By Theorem \ref{r_union}, (3), whenever we extract the union $C'=C_1 \cup C_2$ of two conics, the relative $r'={\rm mrd}(C')=1$, so
by the classification given in \cite[Proposition 5.5]{DIS}, the only possible cases are:
either $|C_1 \cap C_2|=2$ and the two intersection points are two tacnodes for $C'$, or
$|C_1 \cap C_2|=1$ and the singular point is an $A_7$ singularity (the two conics are hyperosculating). Since in the first case the total Tjurina number is $\tau =6$, it does not occur. This settles the case $d=4$. Assume now $d \ge 6$. 

\vskip 2mm
\noindent {\bf Claim: } Whenever we extract the union $C_{i_1} \cup C_{i_2} \cup C_{i_3}$ of three conics, they belong to the same pencil, so they are either bitangent, or they are hyperosculating.

\noindent {\bf Proof of the Claim.}
Indeed, assume first that $C_{i_1}\cap C_{i_2}=\{p,q\}$ and $p, q$ are two tacnodes for $C_{i_1}\cup C_{i_2}$. Without loss of generality we can assume that 
$$
C_{i_1}=V(f_{i_1})=V(xz+a_{i_1}y^2), \qquad
C_{i_2}=V( f_{i_2})=V(xz+a_{i_2}y^2)
$$
with $a_{i_1},a_{i_2}\in \CC$. Therefore, $D_0(f_{i_1})_1=D_0(f_{i_2})_1\cong {\rm Syz}(f_{i_1}f_{i_2})_1=\langle (x,0,-z)\rangle$. By hypothesis ${\rm mrd}(f_{i_1}f_{i_2}f_{i_3})=1$. So, applying Theorem \ref{r_union}, we have 
$$
D_0(f_{i_3})_1\cap \langle (x,0,-z) \rangle \ne {0},\ {\rm or} \
\alpha (x,0,-z)+\beta (x,y,z)\in D_0(f_{i_3})_1.
$$
A straightforward computation shows that necessarily $C_{i_3}=V( f_{i_3})=V(\lambda xz+\mu y^2)$.

Assume now that $C_{i_1}\cap C_{i_2}=\{p\}$ and $p$ is a singularity $A_7$ for $C_{i_1}\cup C_{i_2}$. Without loss of generality we can assume 
$$
C_{i_1}=V (f_{i_1})=V(x^2+a_{i_1}(xz+y^2)), \qquad
C_{i_2}=V( f_{i_2})=V(x^2+a_{i_2}(xz+y^2))
$$ 
with $a_{i_1},a_{i_2}\in \CC$. Therefore, $D_0(f_{i_1})_1=D_0(f_{i_2})_1\cong {\rm Syz}(f_{i_1}f_{i_2})_1=\langle (0,x,-2y)\rangle $. By hypothesis ${\rm mrd}(f_{i_1}f_{i_2}f_{i_3})=1$, hence by Theorem \ref{r_union}, we have 
in this case
$$
D_0(f_{i_3})_1\cap \langle (0,x,-2y)) \rangle \ne {0}, \ {\rm or} \ 
\alpha (0,x,-2y)+\beta (x,y,z)\in D_0(f_{i_3})_1.
$$
A direct computation shows that in this case  necessarily $C_{i_3}=
V(f_{i_3})=V(\lambda x^2+\mu (xz+ y^2))$. 

\vskip 2mm It follows from the claim that the irreducible components of $C=\cup _{i=1}^mC_i$ belong either to a bitangent pencil of conics or to a hyperosculating pencil of conics. In the first case we have $\tau(C)=(2m)^2-3(2m)+2$ (see Example \ref{mintau}(1)) and in the second case it is $\tau(C)=(2m)^2-3(2m)+3$ (see Example \ref{tau_max}(2)), which proves what we want. 

\vskip 2mm
\noindent \underline{Case 3:} : Let $ C =\cup _{i=1}^m C_i \bigcup \cup _{j=1}^s L_j:  \ f=\prod_{i=1}^mf_i \cdot \prod _{j=1}^s \ell _j=0$
 be a conic-line arrangement. By Theorem \ref{r_union}, the conic arrangement $\cup _{i=1}^m C_i$ satisfies ${\rm mrd}(\prod _{i=1}^mf_i)=1$. By the above discussion, the components of $\cup _{i=1}^m C_i$ belong either
 to a hyperosculating pencil or to a bitangent pencil of conics. We analyze this two cases separately. 
 
 Let us first assume that $\cup _{i=1}^m C_i=V(  \prod _{i=1}^m (x^2+a_{i}(xz+y^2)))$ belong to a hyperosculating pencil of conics with hyperosculating point $p=(0:0:1)$. Therefore, ${\rm Syz}(\prod _{i=1}^m f_i)_1=\langle (0,x,-2y)\rangle$. We look for a line $L=V( \ell) =V(ax+by+cz)$ such that ${\rm mrd}((ax+by+cz) \prod _{i=1}^m (x^2+a_{i}(xz+y^2))=1$. By Theorem \ref{r_union}, (3), if we set
 $$
 f_1=ax+by+cz, \quad f_2=\prod _{i=1}^m (x^2+a_{i}(xz+y^2)),
 $$
 we have that either $D_0(f_1) _1 \cap D_0(f_2)_1 \neq 0$, or $\alpha (x\partial_x +y \partial_y +z \partial_z) \in D_0(f_1) _1 + D_0(f_2)_1$ for some nonzero constant $\alpha$. As to the line $V(f_1)$, if $a \neq 0$, we have
 $$
 D_0(f_1) _1= \{ L_1 (-b \partial_x +a \partial_y) + L_2 (-c \partial_x + a \partial_z) \ : \ L_1, L_2 \in R_1\},
 $$
 while $D_0 (f_2)= \langle x\partial_y -2y \partial_z \rangle$.
 
 The condition $D_0(f_1) _1 \cap D_0(f_2)_1 \neq 0$ is never satisfied,
 while
 the condition $\alpha (x\partial_x +y \partial_y +z \partial_z) =L_1 (-b \partial_x +a \partial_y) + L_2 (-c \partial_x + a \partial_z) + \beta x \partial_y - 2\beta y \partial_z$ for $\alpha \neq 0$ has as a unique solution $b=c=0$. It follows that
 $$
 \ell =x.
 $$
 The case $a=0$ can be treated similarly and it never occurs.

 Next it is possible to check in a similar way
 that  
 $h=x(ax+by+cz)\prod _{i=1}^m (x^2+a_{i}(xz+y^2))$ has a linear Jacobian syzygy
 if and only if $b=c=0$, but this would give rise to a non reduced polynomial.

 So, $C$ is as in example \ref{tau_max}.

 Let us now assume that $\cup _{i=1}^m C_i=V( \prod _{i=1}^m (xz+a_{i_1}y^2))$ belongs to a pencil of conics, all of them bitangent at $\{p=(1:0:0),q=(0:0:1)\}$, so that we have the linear syzygy $(x,0,-z)$.
 Arguing as above we can determine the lines that we can add to this conic arrangement in such a way that the new conic-line arrangement has Jacobian ideal with a linear syzygy.
  It turns out that we have only six possibilities:
 $$
 x\prod _{i=1}^m (xz+a_{i_1}y^2), \quad y\prod _{i=1}^m (xz+a_{i_1}y^2), \quad z\prod _{i=1}^m (xz+a_{i_1}y^2), 
 $$
 $$
 xz\prod _{i=1}^m (xz+a_{i_1}y^2), \quad yz\prod _{i=1}^m (xz+a_{i_1}y^2), \quad xyz\prod _{i=1}^m (xz+a_{i_1}y^2).
 $$
 The case $y\prod _{i=1}^m (xz+a_{i_1}y^2)=0$ is not free by Theorem \ref{Conca} (2). By observing that the first and the third case are projectively equivalent, this concludes the proof.
\end{proof}

\subsection{Nearly free conic-line arrangements with
a linear Jacobian syzygy}

In this subsection, we classify conic-line arrangements $C=V(f)$ of degree $d$ in $\PP^2$ with 
$r={\rm mrd}(f)=1$ and minimal Tjurina number $\tau (C)=d^2-3d+2$.

\begin{theorem} \label{taumin}
Let $C=V(f)$ be a conic-line arrangement in $\PP^2$ of degree $d\ge 6$. Then, $\tau(C)=d^2-3d+2$ if and only if 
 $C$ is either a conic arrangement ${\mathcal C}_2$ as in example \ref{mintau}(1), or a conic-line arrangement ${\mathcal C}{\mathcal L}_6$ as in example \ref{mintau}(2).
\end{theorem}
\begin{proof}
 All conic-line arrangements $C\subset  \PP^2$ described in examples \ref{mintau}(1)-(2) have total Tjurina number $\tau(C)=d^2-3d+2$. 
 
 Let us prove the converse. The hypothesis $\tau(C)=d^2-3d+2$, $d\ge 6$ and Proposition \ref{bounds_tau} imply that ${\rm mrd}(f)=1$. By \cite[Proposition 4.3]{DS4}, there are no line arrangements with 
 ${\rm mrd}(f)=1$ and $\tau(C)=d^2-3d+2$. Therefore, we only have two possibilities: either 
 $C$ is a conic arrangement,
or $C$ is a conic-line arrangement.
Arguing as in the proof of Theorem \ref{prop: maxtau} is it possible to conclude.
\end{proof}


\section{Applications}

As application of the previous results we obtain the main result of this paper, namely, we determine when the Jacobian ideal of a conic-line arrangement is the ideal of an eigenscheme. More precisely, we have:

\begin{theorem}\label{main}
Let $C=V(f)$ be a conic-line arrangement in $\PP^2$ of degree $d\ge 4$. The Jacobian ideal $J_f$ of $f$ is the ideal of an eigenscheme $E(T)$ if and only if $C$ is either a line arrangement ${\mathcal L}$,
or a conic-line arrangement ${\mathcal C}{\mathcal L}_2$.
 \end{theorem}

\begin{proof}
Let $C=V(f)$ be a line arrangement (resp. conic-line arrangement) as described in the statement of the theorem. We have seen in example \ref{tau_max}(1) (resp. example \ref{tau_max}(4)) that the Jacobian ideal $J_f$ of $C$ has is defined by the maximal minors of the matrix
$$
\begin{pmatrix}(d-1) x & g_0 \\
-y & g_1\\
-(1+d)z & g_2\\
\end{pmatrix}, \quad
\text {resp. } 
\begin{pmatrix} x & h_0 \\
y & h_1\\
(1-d)z & h_2\\
\end{pmatrix}.
$$
Equivalently, the Jacobian ideal is generated by the minors of the matrix
$$
\begin{pmatrix} x & \frac{1}{d-1} g_0 \\
y & -g_1\\
z & -\frac{1}{d+1}g_2\\
\end{pmatrix}, \quad 
\text {resp. } 
\begin{pmatrix} x & h_0 \\
y & h_1\\
z & \frac {1} {d-1}h_2\\
\end{pmatrix}.
$$
By definition we have $J_f=I(E(T))$, where
$T=(\frac{1}{d-1}g_0,-g_1, -\frac{1}{d+1}g_2)\in (Sym^{d-1}\CC^{3})^{\oplus (3)}$,
resp. $T=(h_0,h_1, \frac{1} {d-1} h_2)\in (Sym^{d-1}\CC^{3})^{\oplus (3)}$ are partially symmetric tensors.

Let us prove the converse. Assume that there is a partially symmetric tensor $T=(g_1,g_2, g_3)\in (Sym^{d-1}\CC^{3})^{\oplus (3)}$ such that $I(E(T))=J_f$. This implies that $C$ is free, $\tau(C)=d^2-3d+3$, ${\rm mrd}(f)=1$ and that ${\rm Syz}(J_f)_1$ is generated by three linearly independent linear forms. Example \ref{tau_max} together with Proposition \ref{prop: maxtau} proves what we want.
\end{proof}

\begin{remark}
There are examples of reduced plane curves $C=V(f)\subset \PP^2$ whose Jacobian ideal $J_f$ is the ideal of an eigenscheme $E(T)$ and they are not conic-line arrangements. For instance, $f=y(x^3-y^2z)$.
\end{remark}

\begin{remark}
In particular, the geometry of the Jacobian scheme of a line arrangement of type ${\mathcal L}$ or a conic-line arrangement ${\mathcal C}{\mathcal L}_2$ is completely described by \cite[Theorem 5.5 and Remark 5.8]{BGV}.
We observe that the cited result concerns only reduced eigenschemes, but it is not difficult to see, that it can be extended to all non reduced zero-dimensional eigenschemes.

Specifically, we have that if $k\in\{2,\dots,d-1\}$ then no subscheme of degree $kd$ of $\Sigma_f$ lies on a curve of degree $k$. Moreover,
the class of $S={\rm Bl}_{\Sigma_f} \PP^2$ in the Chow ring $A(\PP^2 \times \PP^2)$ can be determined. By choosing $L_1$ and $L_2$ as generators of the Picard groups of the two factors, and by setting $p_i:\PP^2 \times \PP^2 \to \PP^2$ to be the two projections, we have that the two divisors
$h_1=p_1^\star L_1$ and $h_2=p_2^\star L_2$
are generators for $A(\PP^2 \times \PP^2)$. Then it is simple to check that
the class of $S$ in $A (\PP^2 \times \PP^2)$ is given by
$$
[S]=(d-1)h_1^2
+ dh_1 h_2 + h_2^2,
$$
and, by taking into account the Hilbert-Burch matrices given in
\eqref{eq: Hilbert Burch L} and
\eqref{eq: Hilbert Burch of $CL$}, the surface
$S$ turns out to be the complete intersection of the two divisors
$T\sim h_1+h_1$ and $D \sim (d-2)h_1+h_2$ given by
$$
T=V
(p_0 x +p_1 y +(1-d)p_2 z), \quad 
D=V(
p_0\partial_{yz}f   +p_1\partial_{xz}f )
$$
in case $\sL$, respectively
$$
T=V((d-1)p_0x -p_1 y-(d+1)p_2 z), \quad
D=V((d-1) \ p_0 \partial_{yz}f
+2p_1\partial_{xz}f
-(d+1) p_2\partial_{xy} f),
$$
in case $\sC \sL_2$,
where $((x:y:z),(p_0:p_1:p_2)) \in \PP^2 \times \PP^2$.

Finally, we observe that by \cite{OO} or \cite[Lemma 5.6]{Abo},
every planar eigenscheme is the zero locus of section $s \in H^0({\mathcal T}_{\PP^2}(d-2))$, where ${\mathcal T}_{\PP^2}$ denotes the
tangent bundle of $\PP^2$.
\end{remark}

Next we shall study the polar map associated with 
${\mathcal L}$ and ${\mathcal C}{\mathcal L}_2$ arrangements.
Observe that since we are concerned with curves of maximal total Tjurina number and quasihomogeneous singularities, the degree of the
generically finite polar map is $(d-1)^2-\mu(C) =d-2$.

\begin{remark}
We can apply the argument used in the proof of \cite[Theorem 5.5]{BGV} and we see that the possible contracted curves by the polar map associated with 
line arrangements of type ${\mathcal L}$ or conic-line arrangements ${\mathcal C}{\mathcal L}_2$ are only lines. 

Indeed, we observe that for any $p=(p_0:p_1:p_2)\notin \Sigma_f$, the point $\nabla f(p)$ is the intersection point of the two distinct lines 
$$
\nabla f(p): \ \left\{
\begin{array}{l}
p_0 x +p_1 y +(1-d)p_2 z=0\\
\partial_{yz}f (p) x +\partial_{xz}f (p)y =0.
\end{array}
\right.
$$
As a consequence, the fiber of $\nabla f$ over any point $q=(q_0:q_1:q_2)\in \PP^2$ is given by the zero locus of
$$
 \left\{
\begin{array}{l}
q_0 x +q_1 y +(1-d)q_2 z=0\\
q_0 \partial_{yz}f +q_1\partial_{xz}f =0.
\end{array}
\right.
$$
Since the first equations represents a line for any choice of $q \in \PP^2$, the claim follows.

\end{remark}

We shall see in the next result that the presence of contracted lines is indeed always confirmed 
for ${\mathcal L}$ and ${\mathcal C}{\mathcal L}_2$
arrangements. Recall that the critical locus of the polar map is given by the hessian curve, and it consists of the contracted curves and the ramification points for the polar map.
\begin{proposition}
Let $C=V(f)$ be a conic-line arrangement in $\PP^2$ of degree $d$ such that the Jacobian ideal $J_f$ of $f$ is the ideal of an eigenscheme $E(T)$.

Then, in case $\sL$, the critical locus of $\nabla f$ is given by an arrangement of $3(d-2)$ lines of the same type
of $\sL$, it contains $\sL$ and the contracted lines by $\nabla f$ are precisely the lines of $\sL$.

In case ${\mathcal C}{\mathcal L}_2$, the critical locus
contains the tangent line $\ell$, and it is the only contracted line.

\end{proposition}

\begin{proof}
It is classically known that all the lines of a line arrangement are contained in the hessian curve. 

Now we verify that the residual curve to ${\mathcal L}$ in $\text {Hess}(f)$ consists of $3(d-2)-d=2d-6$ concurrent lines through $O=(0:0:1)$ and that such residual lines are not contracted by $\nabla f$.

The first claim follows by writing the hessian matrix explicitly:
$$
\hess(f)=\begin{pmatrix} 
\partial_{xx}f & \partial_{xy}f & \partial_{xz} f \\
\partial_{xy}f & \partial_{yy}f & \partial_{yz} f \\
\partial_{xz}f & \partial_{yz}f & 0 \\
\end{pmatrix}.
$$
Since both $\partial_{xz}f$ and $\partial_{yz}f$ are polynomials in $x$ and $y$ only, by developing the determinant $h(f)=\det \hess(f)$ with respect to the last row we see that $\frac{h(f) }{ f}$ is a polynomial in $x$ and $y$.

Moreover, as the polar map is given by
$$
\nabla f= \left( x_2\left( \sum_{i=1}^{d-1} a_i 
\prod_{j\neq i, j=1}^{d-1}(a_j x + b_j y)\right),z\left( \sum_{i=1}^{d-1} b_i 
\prod_{j\neq i, j=1}^{d-1}(a_j x + b_j y)\right),\prod_{i=1}^{d-1}(a_i x + b_i y)
\right),
$$
we see that the line $z=0$ is contracted to the point $(0:0:1)$ and the lines $a_i x +b_i y=0$
to the points $(a_i:b_i:0)$. 

Finally, to prove that there are no other contracted lines, we
recall that the Hilbert-Burch matrix of $J_f$ is given by
\eqref{eq: Hilbert Burch}, and that $\nabla f$ is given by its $2\times 2$ minors. It follows that for any $p=(p_0:p_1:p_2)\notin \Sigma_f$, the point $\nabla f(p)$ is the intersection point of the two distinct lines 
$$
\nabla f(p): \ \left\{
\begin{array}{l}
p_0 x +p_1 y +(1-d)p_2 z=0\\
\partial_{yz}f (p) x +\partial_{xz}f (p)y =0.
\end{array}
\right.
$$
As a consequence, the fiber of $\nabla f$ over a point $q=(q_0:q_1:q_2)\in \PP^2$ is given by the zero locus of
$$
 \left\{
\begin{array}{l}
q_0 x +q_1 y +(1-d)q_2 z=0\\
q_0 \partial_{yz}f +q_1\partial_{xz}f=0.
\end{array}
\right.
$$
In particular, a ramification point appears in a fiber if and only if the set of $d-2$ concurrent lines through $(0:0:1)$ given by the equation $q_0 \partial_{yz}f +q_1\partial_{xz}f =0$ contains a (non reduced) double line, so the question is to determine the non reduced elements of the pencil $q_0 \partial_{yz}f +q_1\partial_{xz}f$. But the latter can be seen as a pencil of divisors in $\PP^1$, and precisely the Jacobian pencil of the polynomial $\partial_2 f$. If the factors of $f$ are general, the polynomial $\partial_2 f \in \CC[x,y]_{d-1}$ is general too. Therefore, the ramification points of the polar map $\nabla \partial_2 f$ are given by its hessian.

We finally treat the case of an ${\mathcal C}{\mathcal L}_2$ arrangement. It is well known that any linear component of a plane curve is contained in the Hessian curve. Moreover, if $f=x\prod _{i=1}^m(xz + a_iy^2)$, the polar map is given by
$$
\nabla f= \left( \prod _{i=1}^m q_i + xz \left(  \sum _{j=1}^m \prod _{i=1, i\neq j}^m  q_i\right), 2xy  \left(  \sum _{j=1}^m a_j\prod _{i=1, i\neq j}^m  q_i\right),    
 x^2 \left(  \sum _{j=1}^m \prod _{i=1, i\neq j}^m  q_i\right)\right),
 $$
 where we set $q_i=xz + a_iy^2$;
we see that the line $V(x)$ is contracted to a point. To see that there are no other contracted lines, we observe that such a line should contain a subscheme of degree at least $d-1$ in
 the Jacobian scheme; the only possible candidates are the tangent line in the second osculating point of the conics, that is the line $V(z)$, or the line $V(y)$ connecting the two singular points; but we can directly check that these cases do not occur.
 
\end{proof}

We conclude by observing that the geometry of the polar map seems to encode some information concerning the topological type of the singularities of a given curve, so we believe that it deserves further investigations.


\begin{thebibliography}{ll}

\bibitem{Abo} {H. Abo},
	{\em On the discriminant locus of a rank $n-1$ vector bundle on
	$\PP^{n-1}$},
	{Portugaliae Mathematica},
	{\bf 77}, n. 3-4
	(2020), {299--343}.

\bibitem{ASS} {H. Abo, A. Seigal and B. Sturmfels},
     {\em Eigenconfigurations of Tensors},
 {Algebraic and Geometric Methods in Discrete Mathematics, Contemporary Mathematics 685},
AMS, Providence, RI,
      (2017), 1--25.



\bibitem {BGLM} E. Artal Bartolo, L. Gorrochategui, I. Luengo, A. Melle-Hern\'andez, {\em On Some Conjectures About Free and Nearly Free Divisors}, in Decker, W., Pfister, G., Schulze, M. (eds), Singularities and Computer Algebra. Springer (2017).

\bibitem{BGV} {V. Beorchia, F. Galuppi, and L. Venturello},
     {\em Eigenschemes of ternary tensors},
  {SIAM J. Appl. Algebra Geom.},
   {\bf 5}
      (2021)
    {n. 4},
      {620--650}.


\bibitem{BC} R.O. Buchweitz and A. Conca, {\em New free divisors from old}, Journal of Commutative Algebra {\bf 5} (2013), 17-47.

\bibitem{D}
A. Dimca, { \em Freeness versus maximal global Tjurina number for plane
              curves},
 {Math. Proc. Cambridge Philos. Soc.},
  {\bf 163} (2017),
{1},
 {161--172}.

 \bibitem {D3} A. Dimca, {\em Jacobian syzygies, stable reflexive sheaves, and Torelli properties for projective hypersurfaces with isolated singularities}, 
Algebraic Geometry 4 (3) (2017), 290-303.
     
\bibitem{D2} A. Dimca, { \em Free and nearly free curves from conic pencils}, J. Korean Math. Soc. 55 (2018), No. 3, pp. 705-717.



\bibitem{DIM} A. Dimca, D. Ibadula and  D.A. Macinic, {\em Numerical invariants and moduli spaces for line arrangements}, 
Osaka J. Math.
{\bf 57} (2020), 847–870.

\bibitem{DIS} A. Dimca, G. Ilardi and G. Sticlaru, {\em Addition-deletion results for the minimal degree of a Jacobian syzygy of a union of two curves}, Journal of Algebra,
{\bf  615} (2023), 77-102.

%
%
%
%

\bibitem{DP} A. Dimca and P. Pokora, {\em On conic-line arrangements with nodes, tacnodes, and ordinary triple points},
Journal of Algebraic Combinatorics, {\bf 56}, n. 2 (2022), 
403--424.

\bibitem {DSer} A. Dimca and E. Sernesi, {\em Syzygies and logarithmic vector fields along plane curves},
 {J. \'{E}c. polytech. Math.},
 {\bf 1},
      (2014),
     {247--267}.
     


\bibitem{DS2} A. Dimca and G. Sticlaru, {\em Free and Nearly Free Curves vs. Rational Cuspidal Plane Curves}, Publ. RIMS Kyoto Univ. 54 (2018), 163-179.

\bibitem{DS4} A. Dimca and G. Sticlaru, {\em 
On the exponents of free and nearly free projective plane curves},
Rev. Mat. Complut. {\bf 30} (2017), 259-268.





\bibitem{DK} I. V. Dolgachev and M. Kapranov. {\em Arrangements of hyperplanes and vector bundles on $\PP^n$}, Duke Math. J. 71 (1993), no. 3, 633-664.

\bibitem{e} D. Eisenbud,
    {\em Commutative Algebra with a view towards Algebraic Geometry},
    Graduate Texts in Mathematics, vol.~150,
    Springer-Verlag, Berlin-Heidelberg-New York, 1995.
    
    \bibitem{ellia} Ph. Ellia, {\em Quasi-complete intersections and global tjurina number of plane curves}, 
    
\bibitem{HE}
  M. Hochster, John A. Eagon, {\em Cohen-{M}acaulay rings, invariant theory, and the generic
       perfection of determinantal loci},
 {Amer. J. Math.},
 {\bf 93},
   (1971),
   {1020--1058}.
   
 \bibitem{L}  L.H. Lim, {\em Singular values and eigenvalues of tensors: a variational approach.} In 1st IEEE International Workshop
on Computational Advances in Multi-Sensor Adaptive Processing, (2005) 129 - 132.          

\bibitem{OO} L. Oeding and G. Ottaviani,
{\em Eigenvectors of tensors and algorithms for Waring decomposition},
Journal of Symbolic Computation,
54 (2013),
9 - 33.




 \bibitem{PW} A.A. du Plessis and C.T.C. Wall, {\em Application of the theory of the discriminant to highly singular plane curves}, Math. Proc. Cambridge Philos. Soc. {\bf 126} (1999), no. 2, 259-266.
   
\bibitem {PW2} A.A. du Plessis and C.T.C. Wall, {\em
Curves in $\PP^2(\CC)$ with $1$-dimensional symmetry}, Rev. Mat. Complut. {\bf 12} (1999), no. 1, 117-132.

\bibitem{Pl} A. P\l oski, {\em A bound for the Milnor number of plane curve singularities}, Cent. Eur. J. Math. 12 (2014), no. 5, 688-693.

\bibitem{Q} L. Qi. {\em Eigenvalues of a real supersymmetric tensor}, Journal of Symbolic Computation {\bf 40} (2005), 1302–1324.

 \bibitem{Sh} J. Shin, {\em A bound for the Milnor sum of projective plane curves in terms of GIT}, J. Korean Math. Soc. 53 (2016), no. 2, 461-473.
\end{thebibliography}
\end{document}